\documentclass{amsart}
\newtheorem{thm}{Theorem}[section]

\newtheorem{cor}[thm]{Corollary}
\newtheorem{lemma}[thm]{Lemma}
\theoremstyle{cupdefn}

\theoremstyle{cuprem}

\numberwithin{equation}{section}

\begin{document}
\title[The Divisor Function at Consecutive Integers]{On a Problem by Erd\H os and Mirsky on the ratio of the number of divisors of consecutive integers}

%% If there is a dedication, include it here
%\dedication{Dedicated to ...}

%\support{Include acknowledgement of support here}

\begin{abstract}
Let $\mathcal{L}$ be the closure of the set of all real numbers $\alpha$, such that there exist infinitely many integers $n$, such that $\alpha=\log\frac{d(n+1)}{d(n)}$, where $d$ is the number of divisors of $n$. We give improved lower bounds for the density of $\mathcal{L}$.
\end{abstract}

\maketitle

\section{Introduction and Results}

For an integer $n$ denote by $d(n)$ the number of divisors of $n$. Erd\H os and Mirsky\cite{EM} asked whether there exist infinitely many $n$ such that $d(n)=d(n+1)$. Building on work of Spiro \cite{Spiro}, this problem was solved by Heath-Brown \cite{HB}. More generally Erd\H os and Mirsky conjectured that the set of all real numbers $\alpha$, such that there exist infinitely many integers $n$ such that $\alpha=\log\frac{d(n+1)}{d(n)}$, is dense in $\mathbb{R}$. This conjecture is still open. Let $\mathcal{L}$ be the closure of the set of all $\alpha$ of this form.

Hildebrand \cite{Hi} showed that the Lebesgue measure $\mathcal{L}_+(x)$ of $\mathcal{L}\cap[0, x]$ is at least $\frac{x}{36}$, and the same bound holds for $\mathcal{L}_-(x)=|\mathcal{L}\cap[-x, 0]|$. These results were improved by Hasanalizade \cite{Ha}, who proved the following.

\begin{thm}[Hasanalizade]
Let $x$ be a positive real number.
\begin{enumerate}
\item We have $\min(\mathcal{L}_+(x), \mathcal{L}_-(x))\geq\frac{x}{3}$.
\item There exists an ineffective constant $A$, such that $\min(\mathcal{L}_+(x), \mathcal{L}_-(x))\geq\frac{x-A}{2}$.
\end{enumerate}
\end{thm}

Here we improve these bounds. We show the following.

\begin{thm}
Let $x$ be a positive real number.
\begin{enumerate}
\item We have $\min(\mathcal{L}_+(x), \mathcal{L}_-(x))\geq\frac{x}{2}$.
\item We have
\[
\limsup_{x\rightarrow\infty}\frac{1}{x}\min(\mathcal{L}_+(x), \mathcal{L}_-(x))\geq\frac{2}{3}.
\]
\end{enumerate}
\end{thm}

\section{Proofs}

The following was proven by Hasanalizade \cite{Ha} using the Goldston-Y\i ld\i r\i m-Pintz-sieve.

\begin{lemma}
Let $a_1, a_2, a_3$ be positive integers. Then there exist indices $1\leq i<j\leq 3$ and infinitely many integers $n$, such that $\frac{d(n+1)}{d(n)}=\frac{a_i}{a_j}$.
\end{lemma}
From this result we deduce the following.
\begin{cor}
Put $\mathcal{N}=\mathbb{R}\setminus\mathcal{L}$. Then there do not exist real numbers $\alpha, \beta, \gamma\in\mathcal{N}$ satisfying $\alpha+\beta=\gamma$.
\end{cor}
\begin{proof}
Suppose $\alpha, \beta, \gamma$ was a counterexample.
As $\mathcal{L}$ is closed, there exists some $r>0$ such that $[\alpha-r, \alpha+r], [\beta-r, \beta+r], [\gamma-r, \gamma+r]$ are contained in $\mathcal{N}$. Pick positive integers $a_1, a_2, a_3$ such that $|\alpha-\log\frac{a_1}{a_2}|<\frac{r}{2}$, $|\beta-\log\frac{a_2}{a_3}|<\frac{r}{2}$. Then $|\gamma-\log \frac{a_1}{a_3}|<r$, and we obtain a contradiction.
\end{proof}

We can now prove the theorem. We only treat the case of $\mathcal{L}_+$, the case of $\mathcal{L}_-$ runs completely parallel. For a set $\mathcal{A}\subseteq\mathbb{R}$ we write $\mathcal{A}(x)$ for $|\mathcal{A}\cap[0, x]|$, and put $\mathcal{A}+\mathcal{A}=\{a+a', a, a'\in\mathcal{A}\}$.

Suppose there exists a positive real number $x$, such that $|\mathcal{N}\cap[0, x]|=\frac{x}{2}+\delta$ with $\delta>0$, and let $x-2\delta<\alpha\leq x$. Then we have
\[
|\mathcal{N}(\alpha)| = |\mathcal{N}(x)|-|\mathcal{N}\cap[\alpha, x]|\geq \frac{x}{2}+\delta - (x-\alpha)>\frac{\alpha}{2},
\]
and by the pigeon hole principle there exists some $\beta>0$, such that $\beta$ and $\alpha-\beta$ are both in $\mathcal{N}$. By the corollary we obtain that $\alpha\not\in\mathcal{N}$.
Hence, $(x-2\delta, x]\cap\mathcal{N}=\emptyset$. But then
$|\mathcal{N}\cap[0, x-2\delta]|=|\mathcal{N}\cap[0, x]|$, and we can repeat the argument until we find $\mathcal{N}\cap[0, x]=\emptyset$. This clearly contradicts the assumption that $|\mathcal{N}\cap[0, x]|>\frac{x}{2}$, and the first claim is proven.

Next suppose that 
\[
\limsup_{x\rightarrow\infty}\frac{1}{x} \mathcal{L}_+(x)<\frac{2}{3}.
\]
Then there exists some $\delta>0$  such that $f(x)=\mathcal{N}(x)-\left(\frac{1}{3}+\delta\right)x$ tends to infinity. As $f$ is continuous, we deduce that it attains its minimum on $[0, \infty)$ in some point $x_0$. We conclude that for any $x>x_0$ we have
\[
|\mathcal{N}\cap[x_0, x]|=f(x)-f(x_0) + \left(\frac{1}{3}+\delta\right)(x-x_0) \geq  \left(\frac{1}{3}+\delta\right)(x-x_0). 
\]
We can now apply Macbeath theorem \cite{Mac}, which is an analogue of Mann's proof of Schnirelmann's conjecture \cite{Mann}.
\begin{lemma}[Macbeath]
Let $\mathcal{A}\subseteq[0, \infty)$ be a measurable set, and assume that $\mathcal{A}(x)\geq\alpha x$ holds for all $x>0$. Then we have $(\mathcal{A}+\mathcal{A})(x)\geq2\alpha x$ for all $x>0$.
\end{lemma}
We apply the lemma to the set $(\mathcal{N}-x_0)\cap[0, \infty)=\{n-x_0:n\in\mathcal{N}, n\geq x_0\}$. We obtain that 
\begin{multline*}
(\mathcal{N}+\mathcal{N})(x) \geq \big((\mathcal{N}\cap[x_0, \infty)) + (\mathcal{N}\cap[x_0, \infty))\big)(x)\\
  = (\big((\mathcal{N}-x_0)\cap[0, \infty)+(\mathcal{N}-x_0)\cap[0, \infty)\big)(x-2x_0)\\
\geq 2(x-x_0)\min_{x_0<t\leq x}\frac{1}{t-x_0}|\mathcal{N}\cap[x_0, t]|\geq\left(\frac{2}{3}+2\delta\right)(x-x_0).
\end{multline*}
By the corollary we have $\mathcal{N}\cap(\mathcal{N}+\mathcal{N})=\emptyset$, hence, 
\[
\mathcal{N}(x)\leq x-(\mathcal{N}+\mathcal{N}(x) \leq \left(\frac{1}{3}-2\delta\right)x+x_0,
\]
which implies that $\mathcal{L}(x)>\frac{2}{3}x$ holds for all $x$ sufficiently large. This contradicts the initial assumption, and the theorem follows.

Both parts of the theorem are sharp in the sense that one cannot deduce better bounds using
 only the fact that $\mathcal{N}$ is sum free. To see this note that 
 $\mathcal{A}=[0, 1)\cup(2, \infty)$ is sum free, and $\mathcal{A}(2)=1$, and 
 $\mathcal{B}=\bigcup_{n\in\mathbb{N}}(n+\frac{1}{3}, n+\frac{2}{3})$ is sum free and 
 satisfies $\mathcal{B}(x)=\frac{x}{3}+\mathcal{O}(1)$.

\end{document}